\documentclass[reqno]{amsart}
\usepackage{amssymb, amsmath}
\usepackage{natbib}
\usepackage{lscape}
\usepackage{dsfont}
\usepackage{mathrsfs}
\usepackage{bm}
\usepackage{color}

\DeclareMathAlphabet{\mathpzc}{OT1}{pzc}{m}{it}

\bibpunct{[}{]}{;}{n}{,}{,}

\newtheorem{te}{Theorem}[section]

\newtheorem{os}[te]{Remark}
\newtheorem{prop}[te]{Proposition}

\numberwithin{equation}{section}

\allowdisplaybreaks

\newcommand {\puntomio} {\mathbin{\vcenter{\hbox{\scalebox{.45}{$\bullet$}}}}}

\def \l { \left( }
\def \r {\right) }
\def \ll { \left\lbrace }
\def \rr { \right\rbrace }

\newcommand{\rev}[1]{{\textcolor{black}{#1}}}

\begin{document}

	\title[]{Time-inhomogeneous jump processes and variable order operators}
	\author{Enzo Orsingher} 
		\email{enzo.orsingher@uniroma1.it}
	\address{Department of Statistical Sciences, Sapienza - University of Rome}
		\author{Costantino Ricciuti} 
		\email{costantino.ricciuti@uniroma1.it}
	\address{Department of Statistical Sciences, Sapienza - University of Rome}
		\author{Bruno Toaldo}
		\email{bruno.toaldo@uniroma1.it}
	\address{Department of Statistical Sciences, Sapienza - University of Rome}
	\date{\today}

		\begin{abstract}
In this paper we introduce non-decreasing jump processes with independent and time non-homogeneous increments. Although they are not L\'evy processes, they somehow generalize subordinators in the sense that their Laplace exponents  are possibly different Bern\v{s}tein functions for each time $t$. By means of these processes, a generalization of subordinate semigroups in the sense of Bochner is proposed. Because of time-inhomogeneity, two-parameter semigroups (propagators) arise and we provide a Phillips formula which leads to time dependent generators.
The inverse processes are also investigated and the corresponding governing equations obtained in the form of generalized variable order fractional equations. An application to a generalized subordinate Brownian motion is also examined.
\end{abstract}

	\maketitle

\tableofcontents

\section{Introduction}
This paper is devoted to the study of non-negative, non-decreasing processes, say $\sigma^\Pi(t)$, with independent and non-stationary increments. We investigate their basic path and distributional properties with particular attention to the governing equations. Such processes are not L\'evy processes since the increments are not stationary and therefore they consist in a generalization of subordinators. The relationship with L\'evy processes is highlighted by the fact that
\begin{align}
\mathds{E} e^{-\lambda \sigma^\Pi (t)} \,= \, e^{-\Pi (\lambda, t)}
\end{align}
where $\lambda \mapsto \Pi (\lambda, t)$ is a Bern\v{s}tein function for all fixed $t \geq 0$, i.e.
\begin{align}
\Pi (\lambda, t) \, = \, &  \lambda b(t) + \int_0^\infty \l 1-e^{-\lambda s} \r \phi(ds, t) 
\end{align}
under suitable assumptions on the couple $\l b(t), \phi(ds, t) \r$. \rev{It turns out that such processes are stochastically continuous and a.s. right-continuous and therefore they are additive processes in the sense of \cite[Definition 1.6]{sato}}. We focus our attention on the case where
\begin{align}
\Pi (\lambda, t) \, = \,  &\int_0^t f(\lambda, s) \, ds \notag \\
= \, & \int_0^t \l \lambda \,   b^\prime(w)  + \int_0^\infty \l 1-e^{-\lambda s} \r\nu(ds, w) \r dw
\label{piint}
\end{align}
where $\lambda \mapsto f(\lambda, s)$ is a Bern\v{s}tein function for each fixed $s \geq 0$. Without the condition of stationarity, the increments $\sigma ^{\Pi} (t)-\sigma ^{\Pi}(s)$ have a distribution $\mu_{s,t}$ depending on both $s$ and $t$, and thus we have to deal with a two-parameter semigroup (also known as propagator). We propose a generalization of the Phillips' formula for subordinate semigroups in the sense of Bochner: if $A$ generates a $C_0$-semigroup of operators $T_t$, $t \geq 0$, on some Hilbert space $\mathfrak{H}$, then the map $t \mapsto \mathcal{T}_{s,t}u =  \int_0^\infty T_wu \, \mu_{s,t} (dw)$ is shown to solve the abstract Cauchy problem
\begin{align}
\frac{d}{dt}q(t) \, = \, -f\l -A,t \r q(t),  \qquad q(s)=u \in \textrm{Dom}(A),
\end{align}
where, for $u \in \textrm{Dom}(A)$,
\begin{align}
-f(-A,t)u \, = \, b^\prime (t)Au + \int_0^\infty \l T_su - u \r \nu(ds, t).
\label{ft}
\end{align}
If $f(x, t) = x^{\alpha (t)}$, for a suitable function $\alpha(t)$ strictly between zero and one, we have
\begin{align}
-f(-A, t) \, = \,- \l -A \r^{\alpha (t)}
\end{align}
and in this sense we say that the operator \eqref{ft} is a \emph{variable order} operator. Furthermore, under suitable assumptions, the one-dimensional marginal of $\sigma^\Pi (t)$ has a density $\mu(x, t)$ satisfying equations of the form
\begin{align}
\frac{\partial}{\partial t} u(x, t) \, = \,- \frac{\partial}{\partial x} \int_0^x u(s,t) \, \bar{\nu}(x-s, t) ds
\label{govsub}
\end{align}
where $\bar{\nu}(s, t) \, = \, \nu((s, \infty), t)$. In the case $\Pi (\lambda, t) = \int_0^t \lambda^{\alpha(s)}ds$, equation \eqref{govsub} becomes
\begin{align}
\frac{\partial}{\partial t} u(x, t) \, = \,- \frac{\partial^{\alpha (t)}}{\partial x^{\alpha (t)}} u(x, t)
\label{18}
\end{align}
where the fractional derivative on the right-hand side must be understood in the Riemann-Liouville sense. The inverse processes
\begin{align}
L^\Pi (t) \, = \, \inf \ll s \geq 0 : \sigma^\Pi (s) > t \rr
\end{align}
are also investigated and the one-dimensional marginal of $L^\Pi (t)$ is shown to solve
\begin{align}
\frac{\partial}{\partial t} \int_0^t u(x,s) \bar{\nu}(t-s,x) ds \,-B_{t,x}u(x, t) = \, -\frac{\partial}{\partial x}u(x, t)
\label{geninv}
\end{align}
where $B_{t,x}$ is an integro-differential operator. Moreover, the composition of a homogeneous Markov process via $ L^\Pi (t) $ is investigated and the corresponding governing equation is derived.

There is a wide literature inspiring such results. First of all there are many papers devoted to the well-known natural relationship between stable subordinators and fractional equations. It is well known that the one-dimensional marginal of a stable subordinator $\sigma^\alpha (t)$ solves the fractional equation $\partial_tu = -\partial_x^\alpha u$. The density of the inverse of $\sigma^\alpha$ solves instead the time-fractional equation $\partial_t^\alpha u = -\partial_xu$. The reader can consult \cite{bazhlekova}, \cite[Chapter 8]{kolokoltsov}, \cite{libro meer} and the references therein for further information on these topics. Note that equation \eqref{18} straigthforwardly generalizes these facts, and furthermore it turns out to be the governing equation of the so-called multistable subordinator studied by \citet{molcha}. When the subordinator has a different Laplace exponent $f$ then the governing equation can be written (see \cite{toaldopota}) as
\begin{align}
\frac{\partial}{\partial t} u (x,t)\, = \,-b\frac{\partial}{\partial x}u(x,t) -\frac{\partial}{\partial x} \int_0^x u(s,t) \bar{\nu}(x-s) ds
\label{subpota}
\end{align}
and \eqref{govsub} is the \emph{variable order} generalization of \eqref{subpota} since the kernel of  the convolution appearing in \eqref{govsub} depends on the time variable. 
Concerning the inverse processes we have a classical result due to \citet{baem} which states that the stochastic solution to a time-fractional Cauchy problem is a L\'evy motion time-changed via the inverse of an $\alpha$-stable subordinator. This framework has been generalized to other subordinators and inverses by different authors. In \citet{meertri} the authors pointed out the the scale limit of a continuous time random walk in $\mathbb{R}^d$ is a L\'evy motion time-changed via the inverse of a general subordinator with Laplace exponent $f(\lambda)$. In this case is proved that the one dimensional marginal solves $f(\partial_t) u= Au$ where $f(\partial_t)u: = \mathcal{L}^{-1} \left[f(\lambda) \widetilde{u}(\lambda)-\lambda^{-1}f(\lambda) u(0) \right](t)$. In \cite{kochu} the author studied operators of the form
\begin{align}
\frac{\partial}{\partial t} \int_0^t u(x,s) k(t-s) ds 
\label{opkoch}
\end{align}
under suitable assumptions on the Laplace transform $\mathcal{L} [k]$. He relates \eqref{opkoch} with the Poisson process time-changed via the inverse of a subordinator, described in \cite{meerpoisson}. In \cite{toaldopota} it is shown that by means of the tail $\bar{\nu}(s)$ of a L\'evy measure $\nu(\puntomio)$ it is possible to write the equation of the inverse of a general subordinator as
\begin{align}
b\frac{\partial}{\partial t} u(x, t) \, + \, \frac{\partial}{\partial t} \int_0^t u(x, s) \, \bar{\nu}(t-s) \, ds \, = \, - \frac{\partial}{\partial x} u(x, t).
\label{tp}
\end{align}
The case in which $\bar{\nu}(s)$ is the tail of the L\'evy measure corresponding to the Laplace exponent $f(\lambda) = \lambda^{2\alpha}+2\mu \lambda^\alpha$, for $\alpha \in (0, 1/2)$ and $\mu>0$, is investigated in detail in \cite{dovetal} and is related to finite velocity random motions. Note that equation \eqref{geninv} generalizes \eqref{tp} in the sense the convolution-type operator on the left-hand side now depends on the space variable. We suggest \cite{bnbook} for a nice summary of the classical theory of time-changed processes. For a different approach to time-changes and the corresponding fractional equations the reader can consult \cite{allouba, orsptrf, orsann}.

An interesting particular case of the processes studied in the present paper is given by the multistable subordinator. Multistable processes provide models to study phenomena which locally look like stable L\'evy motions, but where the stability index evolves in time.
There are two different types of multistable processes (for a complete discussion  see \cite{legueve2}). The first one is the so-called field-based process (see \citep{falco2} \citep{legueve1}), which is neither a markovian nor a pure-jump process. The second one is the  multistable process with independent increments (see \cite{falco1}) with Laplace exponent
\begin{align}
\Pi (\lambda, t) \, = \, \int_0^t \lambda^{\alpha(s)} ds,
\end{align}
which can be considered as the prototype of non-homogeneous subordinators.

\section{Non-homogeneous subordinators}
 Our research concerns the c\'adl\`{a}g processes
\begin{align} 
\sigma^{\Pi}(t) \, = \, b(t) + \sum_{0 \leq s \leq t} \mathpzc{e}(s), \qquad t\geq 0,
\label{definnhs}
\end{align}
where $[0, \infty) \ni t \mapsto b(t)$ is a non-negative, differentiable function such that $b(0) = 0$, and $\mathpzc{e}(s)$ is a Poisson point process in $\mathbb{R}^+ $ with characteristic measure $\nu (dx,dt)$. We will work  throughout the whole paper under the following assumptions
\begin{enumerate}
\item[A1)] $\nu(ds , \puntomio)$ is absolutely continuous with respect to the Lebegue measure, i.e. there exists a density such that $\nu(ds, dt)=\nu(ds, t)dt$. Furthermore the family of measures $\ll \nu(ds, t)  \rr_{t \geq 0}$ is such that the function $t \mapsto  \nu(ds, t)$ is continuous for each $t$.
\item[A2)] for all $t \geq 0$,
\begin{align}
\int_{(0, \infty) \times [0,t]} (x \wedge 1) \nu(dx, s)ds \,  < \, \infty.
\label{12}
\end{align}
\end{enumerate}
We call $\sigma^{\Pi}(t)$, $t>0$, a \emph{non-homogeneous subordinator}. 

Definition \eqref{definnhs} consists in a slight generalization of the L\'evy-It\^{o} decomposition \cite{itodec} which holds for non-decreasing L\'evy processes (subordinators). Therefore $\sigma ^{\Pi} (t)$ retains some important properties of the usual subordinators (that is the increments are independent and the sample paths are non-decreasing) but presents a fundamental difference consisting  in the non-stationarity of the increments (whose distribution is here assumed to be time-dependent).
Hence, the number of points of the poissonian process in any Borel set $\bm{B} \subset \mathbb{R}^+ \times \mathbb{R}^+$ of the form $\bm{B} = B \times [s, t]$, where $B \subset (0, \infty)$, possesses a Poisson distribution with parameter 
\begin{align}
m(\bm{B}) = \int_{\bm{B}} \nu (dx, s)ds \, = \, \int_B \int_{[s,t]} \nu(dx, w) dw.
\end{align}
In particular, the expected number of jumps of size $[x,x+dx)$ occurring up to an arbitrary instant $t$ is given by
\begin{align}
\phi (dx,t) = \int_0^t \nu (dx,\tau)d\tau.
\end{align}
In view of \eqref{12}, which implies that
\begin{align}
\int_0 ^{\infty} (x \wedge 1) \phi(dx, t) \,  < \, \infty \qquad \forall t>0,
\end{align}
we can apply Campbell theorem (see, for example, \cite[p. 28]{king}) to the process (\ref{definnhs}) in order to write that
\begin{align}
\mathds{E} e^{-\lambda \sigma^{\Pi}(t)} \, = \, e^{-\Pi (\lambda, t)}
\end{align}
where
\begin{align}
\Pi (\lambda, t) \, 
= \, & \lambda b(t)+  \int_0^\infty \l 1-e^{-\lambda x} \r  \phi (dx, t).
\label{laplaexp}
\end{align}
Thus the function
\begin{align}
\lambda \mapsto \Pi (\lambda, t) \, = \, \lambda b(t)  +  \int_0^\infty \l 1-e^{-\lambda x} \r \phi(dx, t)
\label{17}
\end{align}
is a Bern\v{s}tein function for each value of $t \geq 0$. We recall that a Bern\v{s}tein function $f$ is defined to be of class $C^\infty$ with $(-1)^{n-1}f^{(n)}(x) \geq 0$, for all $n \in \mathbb{N}$ \cite[Definition 3.1]{librobern}. Furthermore, a function $f$ is a Bern\v{s}tein function if and only if \cite[Theorem 3.2]{librobern}
\begin{align}
f(\lambda) \, = \,a+b\lambda + \int_0^\infty \l 1-e^{-\lambda s} \r \nu(ds)
\end{align}
where $a, b \geq 0$ and $\nu(ds)$ is a measure on $(0, \infty)$ such that
\begin{align}
\int_0^\infty (s\wedge 1) \nu(ds) < \infty.
\end{align}

Note that under A1) and A2), and the further assumption
\begin{align}
\int_0 ^{\infty} (x \wedge 1) \nu(dx, t) \,  < \, \infty,  \qquad \forall t\geq 0,
\label{211}
\end{align}

there exists a Bern\v{s}tein function $\lambda \mapsto f(\lambda , t)$ such that (\ref{17}) can be written as
\begin{align}
\Pi (\lambda, t) \, = \, \int_0^t \l \lambda b^\prime(w)  + \int_0^\infty \l 1-e^{-\lambda s} \r \nu(ds, w) \r dw \, = \, \int_0^t f(\lambda, w) dw.
 \end{align}

In what follows the function $s \mapsto \bar{\nu}(s, t)$ will denote the tail of the measure $\nu(ds, t)$, i.e.
\begin{align}
\gamma \mapsto \bar{\nu}(\gamma,t) \, = \, \nu ((\gamma, \infty),t), \qquad \gamma >0.
\label{taildef}
\end{align}

\subsection{Paths properties}
The process $\sigma^{\Pi}(t)$, $t \geq 0$, is the sum over a Poisson process, hence it has independent increments.
As shown in the following theorems, $\sigma ^{\Pi}(t)$ is continuous a.s. and, under suitable conditions, strictly increasing on any finite interval.

\begin{te}
\label{tecp}
The process $\sigma^{\Pi}(t)$ is a.s. continuous, i.e., $\sigma^\Pi(t) = \sigma^\Pi(t-)$, a.s., for each fixed $t > 0$.
\end{te}
\begin{proof}
Observe that for $h >0$
\begin{align}
\Pr \ll | \sigma^\Pi (t)  - \sigma^\Pi (t-h) | > \epsilon \rr \, = \, & \Pr \ll \sum_{ t-h < s \leq t   } \mathpzc{e}(s)  > \epsilon \rr \notag \\
\leq \, & \Pr \ll \sum_{t-h < s \leq t} \mathpzc{e}(s) \mathds{1}_{\ll \mathpzc{e}(s) <1 \rr} > \frac{\epsilon}{2} \rr \notag \\
& + \Pr \ll \sum_{t-h < s \leq t} \mathpzc{e}(s) \mathds{1}_{\ll \mathpzc{e}(s) \geq 1 \rr} > \frac{\epsilon}{2} \rr.
\label{135}
\end{align}
Now since
\begin{align}
\Pr \ll \sum_{t-h <s \leq t} \mathpzc{e}(s) \mathds{1}_{\ll \mathpzc{e}(s) <1 \rr}> \frac{\epsilon}{2} \rr  \leq \,  &\frac{2}{\epsilon} \mathds{E} \sum_{t-h < s \leq t} \mathpzc{e}(s)  \mathds{1}_{\ll \mathpzc{e}(s) <1 \rr} \notag \\
= \, & \frac{2}{\epsilon}\int_{(t-h,t]} \int_0^1  x\nu(dx, w)dw \stackrel{h \to 0 }{\longrightarrow}0.
\end{align}
For the second term of \eqref{135} we have that
\begin{align}
\Pr \ll \sum_{t-h < s \leq t} \mathpzc{e}(s) \mathds{1}_{\ll \mathpzc{e}(s) \geq 1 \rr} > \frac{\epsilon}{2} \rr \, \leq \,&  1- e^{- \int_1^\infty \int_{(t-h,t]} \nu(dx,w)dw} \notag \\ = \, & 1-e^{- h  \nu((1,\infty), w^\star)} \notag \\
\stackrel{h \to 0}{\longrightarrow} \, & 0.
\end{align}
Continuity in probability implies that for any sequence $t_n \uparrow t$ it is true that there exists a subsequence such that $\sigma^\Pi(t_n) \to \sigma^\Pi (t)$ a.s.. But since the processes $\sigma^\Pi(t)$ are c\'adl\`{a}g the left limit must exist a.s. and therefore it must be equal to $\sigma^\Pi(t)$. Thus the theorem is proved.
\end{proof}
It is well-known that a L\'evy process is strictly increasing on any finite interval if the L\'evy measure is supported on $(0, \infty)$ (and hence is a subordinator) and has infinite mass, $\nu(0, \infty) = \infty$ (see, for example, \cite[Theorem 21.3]{sato}). In our case a similar result is true.
\begin{prop}
\label{remsi}
Let $W$ be a finite interval of $[0, \infty)$. If $b^\prime(t) > 0$ for $t \in W \subseteq [0, \infty)$ the process $\sigma^\Pi (t)$ is a.s. strictly increasing in $W$. If $b^\prime (t) = 0$ for $t \in W$ but $ \nu ((0,\infty),t)= \infty$ for all $t\in W$, then the process $\sigma^{\Pi}(t)$ is a.s. strictly increasing on $W$.
\end{prop}
\begin{proof}
If $b^\prime (w)>0$ for $w \in [s,t]$ then it is clear that the process $\sigma^\Pi(t)$ is strictly increasing. Now let $b^\prime (t) =0$.
Observe that the Poisson point process $\l \mathpzc{e}(s), s \geq 0 \r$ is defined as the only point $\mathpzc{e}(s) \in (0, \infty)$ such that for a Poisson random measure $\varphi (\puntomio)$ on $(0, \infty) \times [0, \infty)$ with intensity $\nu(dx, dt)$ it is true that
\begin{align}
\varphi|_{(0, \infty) \times \ll t \rr}(d\bm{x}) \, = \, \delta_{(\mathpzc{e}(t), t)} (d\bm{x})
\end{align}
where $\bm{x} \in (0, \infty) \times [0, \infty)$. Fix an interval of time $[s,t]$, then the probability that the process $\sigma^\Pi(t)$, $t \geq 0$, does not increase in $[s,t]$ is the probability that there are no points in $(0, \infty) \times [s,t]$. Let $E_j$, $j \in \mathbb{N}$, be a partition of $(0, \infty)$ with $\nu(E_j)< \infty$ for all $j\in\mathbb{N}$, then 
\begin{align}
\varphi ((0, \infty) \times [s, t]) \, = \, \sum_j \varphi_j \l \l (0, \infty) \cap E_j \r \times [s,t] \r
\label{216}
\end{align}
and by the countable additivity Theorem \cite[p. 5]{king} the sum \eqref{216} diverges with probability one since $\sum_j \nu(E_j) = \nu(0, \infty) = \infty$. Since this is true for any interval $[s,t]$ and if $\nu(0, \infty)=\infty$ the theorem is proved.
\end{proof}

\begin{os} \normalfont
\label{rem23}
\rev{The process $\sigma^\Pi (t)$, $t \geq 0$, defined as in \eqref{definnhs} is an additive process in the sense of \cite[Definition 1.6]{sato}. This is because
\begin{enumerate}
\item it has independent increments (it is the sum over a Poisson process);
\item $X_0 = 0$ a.s.;
\item it is stochastically continuous;
\item it is a.s. right-continuous with left-limit (Theorem \ref{tecp}).
\end{enumerate}
Stochastic continuity from the left (Item (3)) can be immediately verified as the stochastic continuity from the right (as in the proof of Theorem \ref{tecp}).
}
\end{os}

\subsection{Time-inhomogeneous random sums}
Since the compound Poisson process is the fundamental ingredient for the costruction of a standard subordinator, it is easy to imagine that random sums with time-dependent jumps play the same role in the definition of non-homogeneous subordinators. 

Let $N(t)$, $t \geq 0$, be a non-homogeneous Poisson process with intensity $g(t)$, $t\geq 0$, and let $T_j= \inf \{t>0 :N(t)=j \}$. We consider the random sum
\begin{align}
\label{compound poisson}
Z(t) = \sum _{j=1}^{N(t)} X(T_j)
\end{align}
where $X(T_{j})$ is the positive-valued jump occurring at time  $T_j$, having the conditional absolutely continuous distribution
\begin{align}
\Pr \{ X(T_j) \in dx | \, T_{j} =t \} = \psi (dx,t), \qquad x\geq 0,
\end{align}
with
\begin{align}
\int _0 ^{\infty} \psi (dx,t)=1, \qquad \forall t>0.
\end{align}
The random variable $Z(t)$ takes the value  $z=0$ with positive probability, and has a density for $z>0$. Indeed
\begin{align}
\Pr \{ Z(t)=0 \}= \Pr \{N(t)=0\}= e^{-\int _0 ^t g(\tau) d \tau }
\end{align}
and, for each $z>0$ we have that
\begin{align}
& \Pr \{ Z(t) \in dz \}\notag\\
=  &\sum _{n=1}^{\infty} \int _0^{t} \int_{t_1}^t \dots  \int _{t_{n-1}}^t 
  \Pr \{ Z(t) \in dz,   T_1 \in dt_1 , T_2 \in dt_2, \dots T_n \in dt_n , N(t)=n\}\notag\\
 = &\sum _{n=1}^{\infty} \int _0^{t} \int_{t_1}^t \dots  \int _{t_{n-1}}^t \Pr \ll \sum _{j=1}^{N(t)}X(T_j) \in dz| \, T_1 = t_1 , T_2 =t_2, \dots T_n = t_n , N(t)=n  \rr \notag \\
 & \times  \Pr  \ll  T_1 \in dt_1 , T_2 \in dt_2, \dots T_n \in dt_n , N(t)=n   \rr \notag\\
 = &\sum _{n=1}^{\infty} \int _0^{t} \int_{t_1}^t  \dots  \int _{t_{n-1}}^t   
  \Pr  \ll \sum _{j=1}^n X(t_j) \in dz \rr   \Pr \ll  T_1 \in dt_1 ,  \dots, T_n \in dt_n , N(t)=n  \rr\notag\\
 = &\sum _{n=1}^{\infty}\frac{1}{n!} \int _{[0,t]^n}
  \Pr  \ll \sum _{j=1}^n X(t_j) \in dz  \rr \Pr \ll  T_1 \in dt_1 , \dots T_n \in dt_n , N(t)=n  \rr.
\end{align} 
The first factor is given by the convolution integral
\begin{align}
 \Pr  \ll \sum _{j=1}^n X(t_j) \in dz  \rr \, = \, dz \int _{(0,\infty)^n} \, \psi (dx_1,t_1) \cdots \psi (dx_n, t_n) \delta \l z-\sum _{j=1}^nx_j \r
\end{align} 
while the second one can be computed as follows
\begin{align}
&\Pr \ll  T_1 \in dt_1 ,  \dots T_n \in dt_n , N(t)\, = \, n \rr \, = \,  g(t_1) \dots g(t_n) e^{-\int_0^t g(\tau) d \tau} dt_1 \dots dt_n.
\end{align}
Then we have
\begin{align}
&\Pr \{ Z(t) \in dz \} \notag  \\
= \, & dz \,  \sum _{n=1}^{\infty}\frac{1}{n!} \int _{[0,t]^n}dt_1 \dots dt_n \int _{(0,\infty)^n} \psi (dx_1,t_1) \cdots \psi (dx_n, t_n) \delta \biggl( z-\sum _{j=1}^nx_j \biggr) \notag \\
&\times e^{-\int_0^t g(\tau) d \tau}  g(t_1) \dots g(t_n).
\end{align}
We define the function
\begin{align}
\label{funzione phi}
\phi (x,t)= \int _0^t g(\tau) \psi (x,\tau) d \tau.
\end{align}
The density of $Z(t)$ can be written as
\begin{align}
\Pr \{ Z(t) \in dz \} = dz \, e^{-\int _0 ^t g(\tau) d \tau} \sum _{n=1}^{\infty} \frac{1}{n!} \int _{(0,\infty)^n}\delta \biggl (z-\sum _{j=1}^nx_j \biggr ) \phi (dx_1,t) \dots \phi (dx_n, t)
\end{align}
and therefore its Laplace transform is given by
\begin{align}
\mathbb{E} e^{-\lambda \, Z(t)}= \Pr \{ Z(t)=0 \}+ \int _{(0, \infty)} e^{-\lambda z} \Pr \{Z(t) \in dz \}
\end{align}
where
\begin{align}
&\int_{(0,\infty)} e^{-\lambda z} \Pr \{Z(t) \in dz \} \, \notag \\
 =  \, & e^{-\int _0 ^{t} g(\tau) d \tau} \sum _{n=1}^{\infty} \frac{1}{n!} \int _{(0,\infty)^n} e^{-\lambda \sum _{j=1}^n x_j} \phi (dx_1,t) \dots \phi (dx_n, t)\notag  \\
= \, & e^{-\int _0 ^{t} g(\tau) d \tau} \sum _{n=1}^{\infty} \frac{1}{n!} \biggl ( \int_{(0 ,\infty)}\, e^{-\lambda x} \phi (dx,t) \biggr )^n \notag  \\
= \, &  e^{-\int _0 ^{t} g(\tau) d \tau}( e^{   \int_{(0,\infty)} \, e^{-\lambda x} \phi (dx,t)  }-1).
\end{align}
Combining all pieces together we have that
\begin{align}
 \mathbb{E} e^{-\lambda Z(t)} &=e^{ -\int _0 ^ t g(\tau) d \tau +   \int_{(0,\infty)}  e^{-\lambda x} \phi (dx,t)} \notag \\
&= e^{-\int _0 ^{t} g(\tau) d \tau \int_{(0,\infty)} \psi(dx,\tau)  +  \int_{(0,\infty)}  \, e^{-\lambda x} \int_0 ^t g(\tau) \psi (dx, \tau) d \tau}   \notag \\
& = e ^{- \int_{(0,\infty)} (1-e^{-\lambda x}) \phi (dx, t) } \label{trasformata di laplace}
\end{align}
with $\phi (x,t)$ of the form \eqref{funzione phi}.
Observe therefore that in this case we have that
\begin{align}
\nu(dx, t) \, = \, \psi (dx, t) g(t) \text{ and } \nu((0, \infty),t) \, = \, g(t) \, < \, \infty.
\end{align}

\subsection{Distributional properties}
We remark that for each $t>0$ the distribution $\mu_t^\Pi(\puntomio)$ of $\sigma ^{\Pi}(t)$ is infinitely divisible. In fact for each $n \in \mathbb{N}$, $\mu _t^{\Pi}(.)$ is given by the $n$-convolution of the probability measure of the r.v.'s associated to the L\'evy exponent 
\begin{align}
\lambda \mapsto \frac{b(t)}{n}\lambda + \int_0^\infty \l 1-e^{-\lambda s} \r \frac{\nu(ds,t)}{n}
\end{align}
where $t$ is fixed. However, unlike what happens for L\'evy processes, such a distribution is not given by the $n$-th convolution of $\mu _{\frac{t}{n}}(\puntomio)$ because the increments are not stationary. \rev{As pointed out in Remark \ref{rem23} the process $\sigma^\Pi (t)$ is an additive process. Therefore the fact that $\mu_t (\puntomio)$ is infinitely divisible is also a consequence of \cite[Theorem 9.1]{sato}}.

It is crucial to observe that  $\sigma ^{\Pi}(t)$ can be approximated by means of a random sum of the form (\ref{compound poisson}), as stated in the following theorem.
\begin{te}
\label{convte}
Let $\sigma ^{\Pi}(t)$ be a non-homogeneous subordinator having Laplace exponent
\begin{align}
 \Pi (\lambda, t) \, = \, \lambda  b(t)   +\int_0^{\infty} (1-e^{-\lambda x}) \phi (dx,t) 
\end{align}
as in \eqref{17} and assume that $s \mapsto \bar{\nu}(s,t)$ is absolutely continuous on $(0, \infty)$ for all $t \geq 0$. Then there exists a process $Z_{\gamma}(t)$ of type (\ref{compound poisson}) such that, for $\gamma \to 0$, we have
\begin{align}
b(t) + Z_{\gamma} (t) \xrightarrow{d} \sigma^{\Pi} (t).
\end{align}
\end{te}
\begin{proof}
Let $\bar{\nu}$ be the function in \eqref{taildef}. Then
\begin{align}
\psi_{\gamma} (dx,t)\, := \, \frac{\nu (dx,t)}{\bar{\nu}(\gamma,t)}  \mathds{1}_{(\gamma , \infty)}(x)
\label{236}
\end{align}
is a probability distribution, because it is positive and integrates to 1. Let us consider the process
\begin{align}
Z_{\gamma}(t)= \sum _{j=1}^{N_{\gamma}(t)}X(T_j)
\end{align}
where $N_{\gamma}(t)$ is a non-homogeneous Poisson process with rate $g_{\gamma}(t)$. We assume $g_{\gamma}(t)= \bar{\nu}(\gamma,t)$ and use \eqref{236} to write
\begin{align}
 \Pr \{X(T_j) \in dx|T_j=t \}= \psi _{\gamma}(dx,t).
\end{align} 
In view of the discussion of Section 2.2 we have that
\begin{align}
p_{Z_\gamma}(dx,t) \,  = \, e^{-\int _0 ^t \bar{\nu}(\gamma, \tau) d \tau} \sum _{n=1}^{\infty} \frac{1}{n!} \phi^{*n}(dx, t) \mathds{1}_{\ll x >  \gamma \rr} + e^{-\int_0^t \bar{\nu}(\gamma, \tau) d\tau} \delta_0(dx)
\end{align}
and
\begin{align}
\mathbb{E}e^{-\lambda  b(t)  -\lambda Z_{\gamma}(t)} \,=  \, e^{-\lambda b(t) -  \int _{\gamma} ^{\infty} (1-e^{-\lambda x}) \int_0^t \psi_\gamma (dx,\tau)g_\gamma(\tau)d\tau }
\end{align}
which converges to $\mathbb{E} e^{-\lambda \sigma^{\Pi}(t)}$ as $\gamma \to 0$.
\end{proof}
Theorem \ref{convte} provides a method to construct an  approximating process. As an example, let's apply such a method to the multistable subordinator defined in \citet{molcha}. In this case, the time-dependent L\'evy measure is
\begin{align}
\phi(dx,t)= dx\int _0^t \frac{\alpha (\tau) x^{-\alpha(\tau) -1}}{\Gamma (1-\alpha (\tau))}d\tau 
\end{align}
for a suitable stability index $\tau \mapsto \alpha (\tau)$ with values in $(0,1)$ in such a way that the conditions A1) and A2) are fulfilled.

The Laplace transform thus reads
\begin{align}
\mathbb{E}e^{-\lambda \sigma^{\Pi}(t)} = e^{-\int _0 ^{\infty} \l 1-e^{-\lambda x} \r \phi (dx,t) }= e ^{-\int _0 ^t \lambda ^{\alpha (s) }ds}
\end{align}
Being
\begin{align}
\bar{\nu} (\gamma ,t)=\int _{\gamma}^{\infty} \nu (dx,t)  = \frac{\gamma^{-\alpha(t) }}{\Gamma (1-\alpha(t)) }
\end{align} 
the approximating process $Z_{\gamma}(t)$ is based on a non-homogeneous Poisson process with intensity $g_{\gamma}(t)= \frac{\gamma^{- \alpha (t)}}{\Gamma (1-\alpha(t)) }$ and jump distribution
\begin{align}
 \psi _{\gamma}(x,t)= \gamma ^{\alpha(t) } \alpha(t)  x^{-\alpha(t) -1}\mathds{1}_{[\gamma, \infty]}(x).
 \end{align}

A convenient way to deal with the non-homogeneity of the multistable process  is to consider its localizability.
We remind that $\sigma^\Pi(t)$ il localizable at $t$ if the following limit holds in distribution (see, for example, \cite{legueve2})): 
\begin{align}
\lim _{r \to 0} \frac{\sigma^\Pi(t+rT)-\sigma^\Pi(t)}{r^{h(t)}}= Z_t(T)
\end{align}
where $Z_t(T)$, $T>0$ is the so-called local process (or tangent process) at time $t$. A fundamental property of $Z_t (T)$ is  $h(t)$-self-similarity, i.e. $Z_t(rT) \stackrel{\text{d}}{=}r^{h(t)}Z_t(T)$ for $r>0$. In the case where $\sigma^\Pi(t)$ is a multistable process, the local approximation at  a fixed $t>0$ is a stable subordinator with index $\alpha (t)$. By taking the Laplace tranform,
\begin{align}
\lim_{r \to 0} \mathbb{E} e^{-\lambda \frac{\sigma^\Pi(t+rT)-\sigma^\Pi(t)}{r^{h(t)}}}= \, & \lim_{r\to 0} \exp  \ll -\int_t^{t+rT} \l \frac{\lambda}{r^{h(t)}} \r^{\alpha (s)}ds  \rr \notag \\
= \, & \lim _{r\to 0} \, \exp \ll -\frac{r T \lambda^{\alpha (t)}}{r^{h(t)\alpha (t)} }+o(r) \rr \notag \\
= \, &e^{-T \lambda ^{\alpha (t)}}
\end{align}
where the limit produces a non-trivial result by assuming that the similarity index is $h (t)= 1/\alpha (t)$. 

Another way to approximate $\sigma^\Pi(t)$ is now given by means of stable processes. We split the interval $[0,t]$ into $n$ sub-intervals of length $\frac{t}{n}$ and assume $\alpha _i = \alpha \l \frac{t}{n}i\r$. We can write
\begin{align}
\mathbb{E} e^{-\lambda \sigma^\Pi(t)} =  e^{-\int _0^t \lambda ^{\alpha (s)}ds}= \lim _{n \to \infty} e^{-\frac{1}{n}\sum _{i=1}^n \lambda ^{\alpha _i} t }= \lim _{n \to \infty} \prod _{i=1}^n e^{-\frac{t}{n} \lambda ^{\alpha _i}}
\end{align}
and this proves that the following equality holds in distribution 
\begin{align}
\sigma^\Pi(t) \, = \, \lim_{n \to \infty} \sum_{i=1}^n \left[ \sigma^{\alpha_i} \l \frac{i}{n}t \r - \sigma^{\alpha_{i-1}} \l \frac{i-1}{n}t \r\right] \,   = \,  \lim _{n \to \infty } \sum_{i=1}^n \sigma^{\alpha_i}  \l \frac{t}{n} \r 
\end{align}
where $\sigma^{\alpha_i}$, $1 \leq i \leq n$, are independent stable subordinators with index $\alpha _i = \alpha \l \frac{t}{n} i \r $.

To conclude this section on distributional properties, we provide sufficient conditions for the absolute continuity with respect to the Lebesgue measure of the distribution of a non-homogeneous subordinator.
\begin{te}
\label{dafare}
Let $\sigma^{\Pi}(t)$, $t \geq 0$ be a non-homogeneous subordinator with Laplace exponent
\begin{align}
\Pi (\lambda, t) \, = \,   \lambda b(t) + \int_0^\infty \l 1-e^{-\lambda s} \r \phi(ds, t) 
\end{align}
Suppose that $s \mapsto \bar{\nu}(s, t)$ is absolutely continuous on $(0, \infty)$ and furthermore assume $\int_0^t \nu((0, \infty), \tau)d\tau=\infty$ for each $t$. The distribution of $\sigma^\Pi(t)$ is absolutely continuous with respect to the Lebesgue measure.
\end{te}
\begin{proof}
Consider the approximating process $Z_\gamma (t)$ with distribution
\begin{align}
p_{Z_\gamma}(dx,t) \,  = \, e^{-\int _0 ^t \bar{\nu}(\gamma, \tau) d \tau} \sum _{n=1}^{\infty} \frac{1}{n!} \phi^{*n}(dx, t) \mathds{1}_{\ll x > \gamma \rr} + e^{-\int_0^t \bar{\nu}(\gamma, \tau) d\tau} \delta_0(dx)
\end{align}
which converges weakly to the law of $\sigma^{\Pi}(t)$ as $\gamma \to 0$ since $Z_\gamma (t) \stackrel{\text{d}}{\longrightarrow} \sigma^\Pi (t)$ as $\gamma \to 0$.
Consider the Lebesgue decomposition of $p_{Z_\gamma}$, written as
\begin{align}
p_{Z_\gamma}(dx, t) \, = \, p_{Z_\gamma}^{\textrm{d}}(dx,t) + p_{Z_\gamma}^{\textrm{s}}  (dx,t) + p_{Z_\gamma}^{\textrm{ac}} (dx,t).
\end{align}
By hypothesis we know that
\begin{align}
p_{Z_\gamma}^{\textrm{d}}(dx,t) \, = \, e^{-\int_0^t \bar{\nu}(\gamma, \tau) d\tau} \delta_0(dx)
\end{align}
since $s \mapsto \bar{\nu}(s, t)$ is absolutely continuous and therefore the L\'evy measure is absolutely continous with respect to the Lebesgue measure.
Therefore by letting $\gamma \to 0$ we observe
\begin{align}
\int_0^\infty \l p_{Z_\gamma}^{\textrm{d}}(dx, t) + p_{Z_\gamma}^{\textrm{s}} (dx, t)\r \, = \,  e^{-\int_0^t \bar{\nu}(\gamma, \tau) d\tau}
\end{align}
which goes to zero as $\gamma \to 0$ since $ \int_0^t \overline{\nu}(0,\tau)d\tau = \infty$ and the continuity of the function $\gamma \mapsto \int_0^t \bar{\nu}(\gamma, \tau) d\tau $ follows from the continuity of $\gamma \mapsto \bar{\nu}(\gamma, t)$ .
\end{proof}

\subsection{The governing equations}
Under the assumptions of the above theorem, we now derive the equation governing the density of a non-homogeneous subordinator. In the following, we will denote as $q(x, t)$ the density of $\sigma^\Pi(t)$, when it exists, i.e.
\begin{align}
\Pr \ll \sigma ^{\Pi}(t) \in dx \rr = q(x,t)dx .
\end{align}
\begin{te}
Let $\sigma^{\Pi}(t)$, $t \geq 0$, be a non-homogeneous subordinator and let the assumptions of Theorem \ref{dafare} hold. Then a Lebesgue density of $\sigma ^{\Pi} (t)$ exists and solves the variable-order equation
\begin{align}
\frac{\partial}{\partial t} q(x, t) \, = \, -b^\prime(t) \frac{\partial}{\partial x}q(x, t) -\frac{\partial}{\partial x} \int_0^x q(s, t) \bar{\nu}(x-s, t)ds, \quad x> b(t), \, t>0,
\label{rl}
\end{align}
provided that $q(x,t)$ is differentiable with respect to $x$, subject to $q(x, 0) dx = \delta_0(dx)$ for $x\geq 0$, and $q(b(t), t) = 0$, for $t >0$.
\end{te}
\begin{proof}
 Now we consider the Laplace transform of the right-hand side of equation \eqref{rl} and we get that
\begin{align}
&\mathcal{L} \left[ b^\prime (t) \frac{\partial}{\partial x}q(x, t) + \frac{\partial}{\partial x} \int_0^x q(x, t) \bar{\nu}(x-s, t) ds  \right] (\lambda) \notag \\ = \, &  \lambda b^\prime (t) \widetilde{q}(\lambda, t) - b^\prime (t) q(b(t), t)+\lambda \, \mathcal{L} \left[ q * \bar{\nu} \right] (\lambda) \notag \\
= \, & \lambda b^\prime (t)\widetilde{q}(\lambda, t)+\lambda \widetilde{q} (\lambda, t) \l  \lambda^{-1} f(\lambda, t) -b^\prime (t)\r - b^\prime (t) q(b(t), t)
\end{align}
where we used the fact that
\begin{align}
\int_0^\infty e^{-\lambda s} \bar{\nu}(s, t) ds\, = \, \frac{1}{\lambda} f(\lambda, t) -b^\prime (t).
\end{align}
Therefore the solution to \eqref{rl} has Laplace transform
\begin{align}
\widetilde{q} (\lambda, t) \, = \, e^{-\lambda b (t)  -b(t) q(b(t),t)- \int_0^t  \int_0^\infty \l 1-e^{-\lambda s} \r \nu(ds, w) dw }
\end{align}
which becomes, since $q(b(t), t) = 0$,
\begin{align}
\widetilde{q} (\lambda, t) \, = \, e^{-\lambda b (t)  - \int_0^t  \int_0^\infty \l 1-e^{-\lambda s} \r \nu(ds, w) dw }
\end{align}
and coincides with $\mathds{E} e^{-\lambda \sigma^{\Pi}(t)}$.
\end{proof}
If $\sigma ^{\Pi}(t)$ is a multistable subordinator with index $\alpha (t)$, we have
\begin{align}
\bar{\nu}(x,t)= \frac{x^{-\alpha (t)}}{\Gamma (1-\alpha (t))},
\end{align}
and the governing equation reads
\begin{align}
\label{equazione governante il multistabile}
\frac{\partial}{\partial t}q(x, t) = -\frac{1}{\Gamma (1-\alpha (t))} \frac{\partial }{\partial x}\int _0^x q(y, t)\frac{1}{(x-y)^{\alpha (t)}} dy.
\end{align}
Keeping in mind the definition of the Riemann-Liouville fractional derivative of order  $\alpha \in (0,1)$ 
\begin{align}
\frac{\partial^{\alpha}}{\partial x^{\alpha}}u(x)= \frac{1}{\Gamma (1-\alpha )} \frac{\partial}{\partial x}\int_0^x \frac{u(y)}{(x-y)^{\alpha }} dy, 
\end{align}
we can write (\ref{equazione governante il multistabile}) as
\begin{align}
\frac{\partial}{\partial t} q(x, t) \,= \, - \frac{\partial^{\alpha(t)}}{\partial x^{\alpha (t)}} q(x, t) \qquad 0<\alpha (t)<1, \, x >0,
\end{align}
where  $\frac{\partial^{\alpha (t)}}{\partial x^{\alpha(t)}}$ is the Riemann-Liouville  derivative of time-varying order $\alpha (t)$.
Then, by taking inspiration from \cite[Definition 2.1]{toaldopota}, we define the generalized Riemann-Liouville derivative with kernel $\overline{\nu}(x,t)$ as
\begin{align}
\mathcal{D}^R_x(t)\;  q(x,t)= \frac{\partial}{\partial x}\int _0 ^x q(s,t) \overline{\nu}(x-s,t)ds 
\end{align}
where the operator $\mathcal{D}^R_x(t)$ acts on the variable $x$ but also depends on $t$. Using this notation, we say that the density of a non-homogeneous subordinator solves the following Cauchy problem
\begin{align}
\begin{cases} \frac{\partial}{\partial t} q(x,t)= - \, \mathcal{D}^R_x (t) \; q(x,t), \qquad t>0,  \\
q(x,0)=\delta(x).
\end{cases}
\end{align}
It is useful to define also a generalization of the Caputo fractional derivative as
\begin{align}
\mathcal{D}^C_x (t) \; q(x,t)= \int _0 ^x \frac{\partial}{\partial s} q(s,t) \overline{\nu}(x-s,t)ds.
\end{align}
If $x \mapsto q(x, t)$ is absolutely continuous on $[0, \infty)$ then $\mathcal{D}^C_x(t)$ exists a.e. for all $t \geq 0$, and the following relationship holds
\begin{align} 
\mathcal{D}^R_x (t)\; q(x,t)= q(0,t)\overline{\nu}(x,t)+\mathcal{D}^C_x (t) \;  q(x,t)
\label{Riemann- Caputo}
\end{align}
whose proof can follow \cite[Proposition 2.7]{toaldopota}. Formula \eqref{Riemann- Caputo} is a generalization of the well-known classical relationship between Caputo and Riemann-Liouville derivatives \cite[page 91]{kill}.

\section{The inverse process}
In this section we consider the process
\begin{align}
L^\Pi (t) \, = \, \inf \ll x \geq 0 : \sigma^\Pi (x) > t \rr
\end{align}
where $\sigma^{\Pi}(x)$ is a non-homogeneous subordinator without drift, namely $b^\prime(x)=0$ for all $x$.
We throughout assume that 
\begin{align}
\nu((0, \infty),t)=\infty \textrm{ for all } t \geq 0
\label{a1}
\end{align}
and that 
\begin{align}\label{a2}
s \mapsto \bar{\nu}(s, t) \, = \, \nu((s, \infty),t)  \textrm{ is an absolutely continuous function on} \, (0, \infty).
\end{align}
By using Theorem \ref{tecp} and Remark \ref{remsi} it is clear that the process $L^\Pi$ is well defined as the inverse process of $\sigma^\Pi(t)$. Observe that, a.s., $L^\Pi(\sigma^\Pi (t)) = t$ since $L^\Pi(\sigma^\Pi (t)) = \inf \ll s \geq  0 : \sigma^\Pi(s) >\sigma^\Pi(t) \rr$ and, under \eqref{a1}, the process $\sigma^\Pi (t)$ is strictly increasing on any finite time interval (Proposition \ref{remsi}). In the following, we denote by $x \mapsto l(x, t)$ the Lebesgue density of $L^\Pi(t)$, when such a density exists. The inverse of a classical subordinator has a Lebesgue density (\cite[Theorem 3.1]{meertri}). We provide here an equivalent version of \cite[Theorem 3.1]{meertri} valid for non-homogeneous subordinators.

\begin{te}
\label{leggeinv}
Under the assumptions \eqref{a1} and \eqref{a2} the process $L^\Pi(t)$, $t\geq 0$, has a Lebesgue density which can be written as
\begin{align} \label{legge dell'inverso}
x \mapsto l (x, t) \, = \, \int_0^t q (s,x) \bar{\nu}(t-s,x)ds, \text{ for all } t > 0.
\end{align}
\end{te}
\begin{proof}
Define
\begin{align}
L(z, t) \, = \, \int_0^z l(x, t) dx
\end{align}
and 
\begin{align}
R(z, t) \, = \, \Pr \ll L^\Pi(t) \leq z  \rr .
\end{align}
We will show that $L(z, t) = R(z, t)$.
By using the convolution theorem for the Laplace transform we have that
\begin{align}
\widetilde{L}(z, \lambda) \, = \, \frac{1}{\lambda}-\frac{1}{\lambda}e^{-\Pi(\lambda, z)}.
\end{align}
The use of the relationship
\begin{align}
\Pr \ll L^\Pi (t) > x \rr \, = \, \Pr \ll \sigma^\Pi (x) <t \rr 
\end{align}
leads to the Laplace transform
\begin{align}
\int_0^\infty e^{-\lambda t} R(z, t) dt \, = \, &\int_0^\infty e^{-\lambda t}\l 1- \Pr \ll \sigma^\Pi (x) < t \rr \r dt \notag \\
= \, & \frac{1}{\lambda} -\frac{1}{\lambda} e^{-\Pi (\lambda, x)}.
\end{align}
Therefore we have proved that
\begin{align}
\int_0^\infty e^{-\lambda t} L(z, t)  dt \, = \, \int_0^\infty e^{-\lambda t} R(z, t) dt
\end{align}
and thus in any point of continuity it is true that
\begin{align}
R(z, t) \, = \, L(z, t).
\end{align}
If we prove that $t \mapsto R(z, t)$ and $t \mapsto L(z,t)$ are continuous functions then we have proved the Theorem for all $t$. Note that under \eqref{a1} the process $\sigma^\Pi (t)$ is strictly increasing on any finite interval in view of Proposition \ref{remsi}. Therefore the process $L^\Pi(t)$ is a.s. continuous and therefore it is also continuous in distribution. This implies that $t \mapsto R(z, t)$ is continuous. Now we show that $t \mapsto L(z, t)$ is continuous. Note that, for $h>0$,
\begin{align}
& \left| l(x, t+h)-l(x,t) \right| \notag \\ 
= \,& \left| \int_0^{t+h} q(s, x) \bar{\nu}(t+h-s,x) ds- \int_0^t q(s, x) \bar{\nu}(t-s,x) ds \right| \notag \\
= \, & \left|  \int_0^t q(s,x) \l \bar{\nu}(t+h-s,x) - \bar{\nu}(t-s,x) \r ds+ \int_t^{t+h} q(s, x) \bar{\nu}(t+h-s,x) ds \right| \notag \\
\leq \, & \int_0^t q(s, x) \left| \bar{\nu}(t+h-s,x) - \bar{\nu}(t-s,x) \right| ds + \int_t^{t+h} q(s, x) \bar{\nu}(t+h-s,x)ds \notag \\
= \, &  \int_0^t q(s, x) \l \bar{\nu}(t-s,x) -\bar{\nu}(t+h-s,x)\r ds + \int_t^{t+h} q(s, x) \bar{\nu}(t+h-s,x)ds.
\label{312}
\end{align}
Since under \eqref{a2} the function $s \mapsto \bar{\nu}(s,\puntomio)$ is absolutely continuous and since
\begin{align}
\bar{\nu}(t-s, x) - \bar{\nu}(t-s+h,x) \, \leq \, \bar{\nu}(t-s, x)
\end{align}
and
\begin{align}
\int_0^t \bar{\nu}(s, x)ds < \infty,
\end{align}
the first integral in \eqref{312} goes to zero by an application of the dominated convergence theorem. The second integral is for any $\infty>z>t$ and sufficiently small $h$
\begin{align}
\int_t^{t+h} q(s, x) \bar{\nu}(t+h-s,x) ds \, = \, \int_t^z q(s, x) \mathds{1}_{\l t, t+h \r}(s) \bar{\nu}(t+h-s, x) ds.
\label{secint}
\end{align}
Now since 
\begin{align}
q(s, x) \mathds{1}_{\l t, t+h \r}(s) \bar{\nu}(t+h-s, x) ds \,  \leq \, q(s, x) \mathds{1}_{\l t, z \r}(s) \bar{\nu}(t-s, x) ds
\end{align}
and 
\begin{align}
\int_t^z q(s, x) \mathds{1}_{\l t, z \r}(s) \bar{\nu}(t-s, x) ds < \infty
\end{align}
another application of the dominated convergence theorem shows that the second integral in \eqref{312} goes to zero. For $h<0$ the arguments are similar. This completes the proof.
\end{proof}

\begin{te}
\label{32te}
If $x \mapsto \bar{\nu}(t, x)$ is differentiable and if the density $x \mapsto l(x, t)$ is differentiable then $l(x, t)$ solves the equation
\begin{align}
\frac{\partial }{\partial x}l(x,t)= \delta (x) \overline{\nu}(t,x) - \mathcal{D}^R_t (x)\; l(x,t)-B_{t,x}l(x,t), \qquad x \geq 0,
\end{align}
in the sense of distributions, namely it solves pointwise the Cauchy problem
\begin{align} \label{equazione inverso}
\begin{cases}
\frac{\partial }{\partial x}l(x,t)=  - \mathcal{D}^R_t (x)\; l(x,t)-B_{t,x}l(x,t) \qquad x>0\\
l(0, t) \, = \, \bar{\nu}(t, 0)
\end{cases}
\end{align}
where $\mathcal{D}^R_t(x)$ is the generalized Riemann-Liouville derivative acting on $t$ and depending on $x \geq 0$, and $B_{t,x}$ is an operator acting on both $t$ and $x$ defined as 
\begin{align}
B_{t,x} \, l(x,t)= \int _0 ^t ds \frac{\partial}{\partial x} \overline{\nu}(t-s,x) \frac{\partial}{\partial s} \int _0 ^x l(x',s)dx^\prime.
\end{align}
\end{te}
\begin{proof}
We can adapt \cite[Theorem 8.4.1]{kolokoltsov} to our case. It is sufficient to derive both sides of (\ref{legge dell'inverso}) and apply
\begin{align} \label{equazione diretto-inverso}
\frac{\partial}{\partial t}l(x,t)=- \frac{\partial}{\partial x}q(t,x),
\end{align} 
to obtain
\begin{align}
\frac{\partial}{\partial x} l (x, t) \, & = \, \int_0^t  \frac{\partial}{\partial x}  q(s,x) \bar{\nu}(t-s,x)ds +   \int_0^t q(s,x) \frac{\partial}{\partial x} \bar{\nu}(t-s,x)ds \notag \\
& = - \, \int_0^t  \frac{\partial}{\partial s}  l(x,s) \bar{\nu}(t-s,x)ds - \int _0 ^t ds \frac{\partial}{\partial x} \overline{\nu}(t-s,x) \frac{\partial}{\partial s} \int _0 ^x l(x',s)dx' \notag \\
&=- \mathcal{D}^C_t (x)\; l(x,t)- \int _0 ^t ds \frac{\partial}{\partial x} \overline{\nu}(t-s,x) \frac{\partial}{\partial s} \int _0 ^x l(x',s)dx' \notag \\
& = \delta (x) \overline{\nu}  (t,x) - \mathcal{D}^R_t (x)\; l(x,t)- \int _0 ^t ds \frac{\partial}{\partial x} \overline{\nu}(t-s,x) \frac{\partial}{\partial s} \int _0 ^x l(x',s)dx' \notag \\
\end{align}
where in the last step we referred to (\ref{Riemann- Caputo}).
\end{proof}

\begin{os} \normalfont
Non stationarity is here expressed by the term $B_{t,x}l(x,t)$,  which vanishes in the case of the inverse of a classical subordinator, and by the fact that the kernel of $\mathcal{D}^R_t(x)$ depends on $x\geq 0$.
\end{os}
In the case of the inverse of a classical stable subordinator, Theorem \ref{32te} obviously leads to the well-known Cauchy problem \cite[eq (5.7)]{meerstra}
\begin{align}
\begin{cases}
\frac{\partial}{\partial x} l(x, t) \, = \, - \frac{\partial^\alpha}{\partial t^\alpha}l(x, t), \qquad x >0, t>0 \\
l(0, t) \, = \, \frac{t^{-\alpha}}{\Gamma (1-\alpha)}, 
\end{cases}
\end{align}
for $\alpha \in (0,1)$.

We remark that a first study on time fractional equations with state-dependent index appears in \cite{garra} where the following equation is considered
\begin{align}
\frac{d^{\alpha(k)}}{d t^{\alpha(k)}} p_k(t) \, = \, -\theta \l p_k(t) - p_{k-1}(t) \r.
\end{align}
The fractional derivative is meant in the Dzerbayshan-Caputo sense and the second member arises by writing the forward equations with the adjoint of the generator of the Poisson semigroup.

\subsection{Time changed Markov processes via the inverse  of non-homogeneous subordinators.}

We now consider the composition of a Markov process with the inverse of a non-homogeneous subordinator. Let $X(u)$, $u>0$ be a Markov process in $\mathbb{R}^d$ such that $X(0)=y$ a.s. and 
\begin{align}
\Pr \{ X(u) \in dx \} = p(x,y,u) dx.
\end{align} 
We assume that $p(x,y,u)$ is a smooth probability density satisfying the following Cauchy problem:
\begin{align} 
\begin{cases}\frac{\partial }{\partial u}p= S_x p, \qquad u>0, \\
p(x,y,0)=\delta (x-y),
\label{equazione processo markoviano}
\end{cases}
\end{align}
where $S_x$ is the adjoint of the Markovian generator acting on the variable $x$.
 Moreover let $L^{\Pi}(t)$, $t\geq 0$, be the inverse of a non-homogeneous subordinator, with density as in Theorem \ref{leggeinv}
 \begin{align}
 \Pr \{ L^{\Pi}(t) \in dx\}= l(x,t)dx.
 \end{align}
 
By assuming that $X(t)$ and $L^{\Pi}(t)$ are independent, we study the composition $X(L^{\Pi}(t))$, having distribution
 \begin{align}
  \Pr \{ X(L^{\Pi}(t)) \in dx \} = \int_0^{\infty} \Pr \{ X (u) \in dx \} \Pr \{ L^{\Pi}(t) \in du \}.
 \end{align}
Then $X(L^{\Pi}(t))$ has a smooth density, defined as
\begin{align}
g(x,y,t)=\int _0^{\infty} p(x,y,u) l(u,t) du. 
\label{densita gg}
\end{align}
By using simple arguments, we now derive the governing equation for (\ref{densita gg}).

\begin{prop}
Under the above assumptions, the density (\ref{densita gg}), for $t \geq 0$, solves the following equation in the sense of distributions:
\begin{align}
\int_0^\infty \mathcal{D}^R_t (u) \, \left[ p(x, y, u) \, l(u, t) \right] du \, = \, & \delta (y-x) \overline{\nu}(t,0)+S_xg(x,y,t) \notag \\
& -\int _0^{\infty} p(x,y,u) B_{t,u}\, l(u,t)du.
\end{align}
\end{prop}

\begin{proof}
We have
\begin{align}
\int_0^\infty \mathcal{D}^R_t (u) \;  \left[ p(x, y, u) \, l(u, t) \right] du = \int _0^{\infty} p(x,y,u)\mathcal{D}^R_t (u)\; l(u,t)du
\end{align}
and by using (\ref{equazione inverso}) and (\ref{equazione processo markoviano}), which hold for positive times, we can write
\begin{align}
 & \int_0^\infty \mathcal{D}^R_t (u) \; \left[ p(x, y, u) \, l(u, t) \right] du \notag \\   = \, & -\lim _{\epsilon \to 0} \int _{\epsilon}^{\infty} p(x,y,u)\frac{\partial}{\partial u}l(u,t) du- \lim _{\epsilon \to 0} \int _{\epsilon} ^{\infty} p(x,y,u) B_{t,u}l(u,t)du \notag \\
=  \, & -\lim _{\epsilon \to 0}\, [p(x,y,u)l(u,t)]_{\epsilon}^{\infty}+\lim _{\epsilon \to 0} \int _{\epsilon}^{\infty}\frac{\partial }{\partial u}p(x,y,u) l(u,t)du \notag \\ & \qquad - \lim _{\epsilon \to 0} \int _{\epsilon} ^{\infty} p(x,y,u) B_{t,u}l(u,t)du\notag\\
 = \, & \delta (y-x) \overline{\nu}(t,0)+\lim _{\epsilon \to 0} \int _{\epsilon}^{\infty}S_xp(x,y,u) \, l(u,t)du -\int _0^{\infty} p(x,y,u) B_{t,u} l(u,t)du\notag\\
 = \, & \delta (y-x) \overline{\nu}(t,0)+S_xg(x,y,t) -\int _0^{\infty} p(x,y,u) \,B_{t,u}\, l(u,t)du
\end{align}
and the proof is complete.
\end{proof}

\begin{os} \normalfont
In the case $X(t)$ is a Brownian motion starting from $y$ and $L^{\Pi}(t)$ is the inverse of a multistable subordinator with index $\alpha (x) \in (0,1)$ we have $\overline{\nu}(t,x)=\frac{t^{-\alpha (x)}}{\Gamma (1-\alpha (x))}$ and $\mathcal{D}^R_t(x)= \frac{\partial ^{\alpha (x)}}{\partial t ^{\alpha (x)}}$, 
and thus the governing equation reads
\begin{align}\label{eq diff fraz}
\int_0^\infty \frac{\partial^{\alpha (u)}}{\partial t^{\alpha (u)} } \, \left[ p(x, y, u) \, l(u, t) \right] \, du \, = \, & \frac{1}{2}\Delta _x g(x,y,t) +\delta (y-x)\frac{t^{-\alpha_0}}{\Gamma (1-\alpha _0)} \notag \\ & - \int _0^{\infty} \frac{1}{\sqrt{2\pi u}} e^{-\frac{(y-x)^2}{2u}}B_{t,u}\, l (u,t)\, du \qquad x\geq 0
\end{align}
where $\alpha (0)= \alpha _0$ and 
\begin{align}
B_{t,u}\, l(x, t)&=  \int _0^t ds \biggl [ \frac{\partial }{\partial u} \overline{\nu} (t-s,u)\frac{\partial}{\partial s}\int _0 ^{u} l(w,s)dw \biggr ] \notag\\
&=   \int _0^t ds  \biggl [\frac{\partial }{\partial u} \frac{(t-s)^{-\alpha (u)}}{\Gamma (1-\alpha (u))}\frac{\partial}{\partial s}\int _0 ^{u} l(w,s)dw \biggr ]
\end{align}
Note that (\ref{eq diff fraz}) is a generalization of the well-known fractional diffusion equation to which it reduces when $u \mapsto \alpha (u)$ is constant, that is
\begin{align}
\frac{\partial^\alpha}{\partial t^\alpha}g \, -\frac{t^{-\alpha}}{\Gamma(1-\alpha)} \delta (x-y) \, = \, \Delta_x g 
\end{align} 
\end{os}

\begin{os} \normalfont
Consider the case where the Markov process is a deterministic time, namely the starting point is $y=0$ and $X(t)=t$. In this case we have $S_x= -\frac{\partial }{\partial x}$ so that the governing equation becomes, for $x \geq 0$,
\begin{align}
\int_0^\infty \mathcal{D}^R_t (u) \, \left[ p(x, y, u) \, l(u, t)\right] du \,  = \delta (x) \overline{\nu}(t,0)-\frac{\partial}{\partial x} g -\int _0^{\infty} \delta (x-u) B_{t,u}\, l(u,t)du
\end{align}
and obviously coincides with that of $L^\Pi$ since the probability density of $X(u)$ is $p(x,0,u)= \delta (x-u)$.
\end{os}

\section{Non-homogeneous Bochner subordination}
We consider in this section a generalization of the Bochner subordination. We recall here some basic facts.
Let $T_t$ be a $C_0$-semigroup of operators (the reader can consult \cite{kolokoltsov} for classical information on this topic) i.e. a family of linear operators on a Banach space $\l \mathfrak{B}, \left\| \puntomio \right\|_{\mathfrak{B}} \r$ such that, for all $u \in \mathfrak{B}$,
\begin{enumerate}
\item $T_0u = u$
\item $T_t T_s u = T_{t+s}u$, $s, t \geq 0$,
\item $\lim_{t \to 0}\left\| T_t u - u \right\|_{\mathfrak{B}} = 0$.
\end{enumerate}
Let $\l A, \textrm{Dom}(A) \r$ be the generator of $T_t$, i.e. the operator
\begin{align}
Au:= \,  \lim_{t \to 0} \frac{T_tu-u}{t}
\end{align}
defined on
\begin{align}
\textrm{Dom} (A) \, = \, \ll u \in \mathfrak{B}: \lim_{t \to 0} \frac{T_tu-u}{t} \textrm{ exists as strong limit} \rr
\end{align}
and let $\left\| T_tu  \right\|_{\mathfrak{B}} \leq \left\| u \right\|_{\mathfrak{B}}$.

Let $\mu_t(\puntomio)$ be a convolution semigroup of sub-probability measures associated with a subordinator, i.e. a family of measures $\ll \mu_t \rr_{t \geq 0}$ satisfying
\begin{enumerate}
\item $\mu_t(0, \infty) \leq 1$, for all $t \geq 0$,
\item $\mu_t  * \mu_s = \mu_{t+s}$, 
\item  $\lim_{t \to 0} \mu_t = \delta_0$ vaguely,
\end{enumerate}
and such that
\begin{align}
\mathcal{L}\left[ \mu_t \right] (\lambda) \, = \, e^{-tf(\lambda)},
\end{align}
where
\begin{align}
f(\lambda) \, = \, a + b\lambda + \int_0^\infty \l 1-e^{-\lambda s} \r \nu(ds)
\end{align}
is a Bern\v{s}tein function.
The operator defined by the Bochner integral
\begin{align}
T_t^fu \, = \, \int_0^\infty T_su \, \mu_t(ds), \qquad u \in \mathfrak{B},
\end{align}
is said to be a subordinate semigroup in the sense of Bochner. A classical result due to \citet{phillips} states that $T_t^f$ is again a $C_0$-semigroup and is generated by
\begin{align}
-f(-A)u:= -au +bAu + \int_0^\infty \l T_s u -u \r \nu(ds)
\end{align}
which is always defined at least on $\textrm{Dom}(A)$ \cite[Theorem 12.6]{librobern}.

In order to extend such a result to non-homogeneous evolutions, a generalization of the notion of one-parameter semigroup is needed. 
Let $(\mathfrak{B},\left\|.\right\|_{\mathfrak{B}})$ be a Banach space. A family of mappings $\mathcal{T}_{s,t}$ from $\mathfrak{B}$ to itself, defined by the pair of numbers $s$ and $t$ (such that $0\leq s \leq t$), is said to be a propagator (two-parameter semigroup) if for each $u \in \mathfrak{B}$, \cite[Section 1.9]{kolokoltsov}
\begin{enumerate}
\item $\mathcal{T}_{t,t}u=u$, for each $t\geq 0$;
\item $\mathcal{T}_{s,t}\mathcal{T}_{r,s}u= \mathcal{T}_{r,t}u$,  for $r\leq s \leq t$;
\item $\lim _{\delta \to 0} \left\|\mathcal{T}_{s+\delta,t}u- \mathcal{T}_{s,t}u\right\|_{\mathfrak{B}} \, = \, \lim _{\delta \to 0}   \left\|\mathcal{T}_{s,t+\delta}u- \mathcal{T}_{s,t}u\right\|_{\mathfrak{B}}= 0  $;
\end{enumerate}
It is obvious that a propagator $\mathcal{T}_{s,t}$ reduces to a classical one-parameter semigroup in the case where it only depends on the difference $t-s$.

Let $\sigma^\Pi (t)$, $t\geq 0$,  be a non-homogeneous subordinator and consider the measures $\mu _{s,t}(.)$ corresponding to the distribution of the increments $\sigma^\Pi(t) - \sigma^\Pi(s)$ which are obviously such that
\begin{align}
\mathcal{L}[\mu _{s,t}](\lambda)= e^{-\int _s^tf(\lambda, \tau)d\tau}
\label{laplst}
\end{align}
as can be ascertained by applying the Campbell theorem to $\sigma^\Pi (t) - \sigma^\Pi (s)$ under the assumption \eqref{211}. Therefore, it is easy to verify that the family of measures $\ll \mu_{s, t} (\puntomio) \rr_{0 \leq s \leq t}$ forms a two-parameter convolution semigroup of probability measures since, from the independence of the increments and \eqref{laplst}, we get $\mu_{s, t} * \mu_{r, s} = \mu_{r, t}$, $r \leq s \leq t$. Consider the operator defined by the Bochner integral on $\mathfrak{B}$
\begin{align}
\mathcal{T}_{s,t}u = \int _0^{\infty}T_{\omega} u \, \mu _{s,t}(d\omega).
\label{defpr}
\end{align}
The family of operators $\ll \mathcal{T}_{s,t} \rr_{0 \leq s \leq t}$ forms a two-parameter semigroup of operators on $\mathfrak{B}$, i.e., \eqref{defpr} is a propagator. This can be easily ascertained by observing that for all $u \in \mathfrak{B}$
\begin{align}
\mathcal{T}_{s,t}\mathcal{T}_{r,s}u \, = \, &\int_0^\infty T_w \left[ \int_0^\infty T_{w^\prime}u \, \mu_{r,s}(dw^\prime) \right] \, \mu_{s, t}(dw) \notag \\
= \, & \int_0^\infty \int_0^\infty T_{w+w^\prime}u \,  \mu_{r,s}(dw^\prime) \, \mu_{s,t}(dw) \notag \\
= \, & \int_0^\infty \int_w^\infty T_{\rho} u\, \mu_{r, s}(d(\rho-w)) \, \mu_{s, t}(dw) \notag \\
= \, & \int_0^\infty T_\rho u \int_0^\rho  \mu_{s,t}(d(\rho-w))  \, \mu_{r, s}(dw) \notag \\
= \, & \int_0^\infty T_\rho u \, \mu_{r,t}(d\rho)\notag \\
= \, & \mathcal{T}_{r,t}u.
\end{align}

We consider here the case where the generator $\l A, \textrm{Dom}(A) \r$ of $T_t$ is a self-adjoint, dissipative operator on an Hilbert space $\l \mathfrak{H}, \langle \puntomio, \puntomio \rangle \r$ and thus we have that $\left\| T_tu \right\|_{\mathfrak{H}} \leq \left\|u \right\|_{\mathfrak{H}}$ (see, for example, \cite[Section 2.7]{jacob1} and \cite[Chapter 11]{librobern} for classical information on linear operators on Hilbert spaces). Recall that an operator is said to be dissipative if $\langle Au,u \rangle \leq 0$ for all $u \in \textrm{Dom}(A)$ and
\begin{align}
\textrm{Dom}(A) \, = \,  \ll u \in \mathfrak{H}: \left\| A u \right\|_{\mathfrak{H}}^2 < \infty \rr.
\end{align}

\begin{te}
\label{teo41}
Let the above assumptions (including \eqref{211}) be fulfilled. The family of operators $\mathcal{T}_{s,t}$ acting on an element $u \in \mathfrak{H}$ is a bounded propagator on $\mathfrak{H}$ and for $u \in \textrm{Dom}(A)$, the map $t \mapsto\mathcal{T}_{s,t}u$ solves, 
\begin{align}
\begin{cases}
\frac{d}{dt} q(t) \, = \, -f(-A,t) q(t), \qquad 0 \leq s \leq t,\\
q(s) = u \in \textrm{Dom}(A),
\end{cases}
\end{align}
where the family of generators $\ll -f(-A,t) \rr_{t \geq 0}$, can be defined as
\begin{align}
-f(-A,t) q\, := \, b^\prime (t) A q + \int_0^\infty \l T_s q -q \r \nu(ds, t), 
\label{deffa}
\end{align}
a Bochner integral on $\textrm{Dom}(A)$.
\end{te}
\begin{proof}
First note that
\begin{align}
\left\| \mathcal{T}_{s,t} u \right\|_{\mathfrak{H}} \, \leq \, \int_0^\infty \left\| T_wu \right\|_{\mathfrak{H}} \mu_{s,t}(dw) \,  \leq \left\| u \right\|_{\mathfrak{H}},
\end{align}
and therefore $\mathcal{T}_{s,t}$ is bounded.

Then we recall (\cite[Theorem 11.4]{librobern} and \cite[Theorem 2.7.30]{jacob1}) that within such a framework we have by the spectral Theorem that for $u \in \textrm{Dom}(A)$
\begin{align}
Au \, = \, \int_{(-\infty, 0]} \lambda \, E(d\lambda)u
\label{314}
\end{align}
where $E(B): \textrm{Dom}(A) \mapsto \textrm{Dom}(A)$, $B$ a Borel set of $\mathbb{R}$, is an orthogonal projection-valued measure supported on the spectrum of $A$ defined as
\begin{align}
E(B) u := \int_B E(d\lambda)u.
\end{align}
Therefore since from \eqref{314} it is true that for a function $\Phi:(-\infty, 0] \mapsto \mathbb{R}$
\begin{align}
\Phi (A)u \, = \, \int_{(-\infty, 0]} \Phi(\lambda) E(d\lambda)u
\label{418}
\end{align}
we have that
\begin{align}
T_tu \,   = \,  \int_{(-\infty, 0]} e^{t\lambda} E(d\lambda)u.
\end{align}

We now verify that $\mathcal{T}_{s, t} \mathcal{T}_{r, s}u= \mathcal{T}_{r,t}u$, $r \leq s \leq t$, in order to show that for all $u \in \mathfrak{H}$ the operator $\mathcal{T}_{s,t}$ is a propagator since the other defining properties as trivially verified. We have that, for all $u \in \mathfrak{H}$,
\begin{align}
\mathcal{T}_{s, t} \mathcal{T}_{r, s}u \, = \,& \int_0^\infty \int_0^\infty T_{w+\rho}u \,  \mu_{r, s} (dw) \mu_{s,t}(d\rho) \notag \\
= \, & \int_0^\infty \int_0^\infty  \int_{(-\infty, 0]} \int_{(-\infty, 0]} e^{\lambda w}e^{\varrho \rho} E(d\lambda) E(d\varrho)u \, \mu_{r, s} (dw) \mu_{s,t}(d\rho)\notag \\
= \, & \int_{(-\infty, 0]} \int_{(-\infty, 0]} e^{-\int_r^s f(-\lambda, w) dw} e^{-\int_s^t f(-\varrho, w) dw} E(d\lambda) E(d\varrho) u \notag \\
= \, & \int_{(-\infty, 0]} e^{-\int_r^t f(-\lambda, w) dw} E(d\lambda)u \notag \\
= \, & \int_0^\infty \int_{(-\infty, 0]} e^{\lambda w} \mu_{r,t}(dw) E(d\lambda) u \notag \\
= \, & \int_0^\infty T_wu \,  \mu_{r,t}(dw) \notag \\
= \, & \mathcal{T}_{r,t}u.
\end{align}
For a function $u$ such that
\begin{align}
\int_{(-\infty, 0]} \left| f(-\lambda, t) \right|^2  \langle E(d\lambda) u, u \rangle < \infty
\end{align}
the representation \eqref{deffa} can be shown to be true: use \eqref{418} to write 
\begin{align}
-f(-A, t) u\, = \,&- \int_{(-\infty, 0]} f(-\lambda, t) E(d\lambda)u \notag \\
 = \,& -\int_{(-\infty, 0]}  \l - b^\prime (t) \lambda + \int_0^\infty \l 1-e^{\lambda s} \r \nu(ds, t) \r \, E(d\lambda) u \notag \\
 = \, &\int_{(-\infty, 0]} b^\prime (t) \lambda E(d\lambda)u + \int_0^\infty \int_{(-\infty, 0]} \l e^{\lambda s} -1 \r E(d\lambda) u  \, \nu(ds, t)  \, \notag \\
 = \, &b^\prime (t) Au + \int_0^\infty  \l T_s u-u  \r\nu(ds, t).
 \label{422}
\end{align}
Now we show that \eqref{422} is true for any $u \in \textrm{Dom}(A)$
\begin{align}
\left\| f(-A,t) u \right\|_{\mathfrak{H}} \, \leq \,& b^\prime(t) \left\| Au \right\|_{\mathfrak{H}} + \int_0^\infty \left\| T_s u - u  \right\|_{\mathfrak{H}} \, \nu(ds,t) \notag \\
\leq \, & b^\prime (t) \left\|Au \right\|_{\mathfrak{H}} + \int_0^1 s \left\| Au \right\|_{\mathfrak{H}} \nu(ds,t) +2 \int_1^\infty \left\| u \right\|_{\mathfrak{H}} \nu(ds,t).
\label{domfa}
\end{align}
Now note that
\begin{align}
\mathcal{T}_{s,t} u \, = \, &\int_0^\infty T_wu \, \mu_{s,t}(dw) \notag \\
= \, & \int_0^\infty \left[ \int_{(-\infty, 0]} e^{w\lambda} E(d\lambda)u \right]\, \mu_{s,t}(dw) \notag \\
= \, & \int_{(-\infty, 0]} e^{-\int_s^t f(-\lambda, \tau)d\tau} E(d\lambda)u
\label{reprtp}
\end{align}
where we used \eqref{418}.
The fact that $\mathcal{T}_{s,t}$ maps $\textrm{Dom}(A)$ into itself can be ascertained by using again \cite[Theorem 11.4]{librobern} for saying that $E(\puntomio)$ maps $\textrm{Dom}(A)$ into itself and furthermore, since $E(I) E(J) = E(I\cap J)$ for any $I,J$ Borel sets of $\mathbb{R}$, we observe that for any $u \in \textrm{Dom}(A)$
\begin{align}
\mathcal{T}_{s,t} A u \, = \, &\int_{(-\infty, 0]} e^{-\int_s^t f(-\lambda, w)dw} E(d\lambda) \int_{(-\infty, 0]} \mu E(d\mu) u \notag \\
= \, & \int_{(-\infty, 0]} \lambda e^{-\int_s^t f(-\lambda, w)dw}E(d\lambda)u \notag \\
= \, & \int_{(-\infty, 0]}  \mu E(d\mu)\int_{(-\infty, 0]} e^{-\int_s^t f(-\lambda, w)dw} E(d\lambda)u\notag \\
= \, & A\mathcal{T}_{s,t} u .
\end{align}
Now note that the equality
\begin{align}
\frac{d}{dt} \mathcal{T}_{s,t} u \, = \,-f(- A,t) \mathcal{T}_{s,t} u, \qquad 0 \leq s \leq t,
\end{align}
must be true in the sense of \eqref{418} and indeed by using \eqref{reprtp} we have that, for $u \in \textrm{Dom}(A)$,
\begin{align}
\frac{d}{dt} \mathcal{T}_{s,t} u \, = \, &\int_{(-\infty, 0]} \frac{d}{dt} e^{-\int_s^t f(-\lambda, w )dw}  E(d\lambda)u \notag \\
= \, & -\int_{(-\infty, 0]} f(-\lambda, t)e^{-\int_s^t f(-\lambda, w) dw} E(d\lambda)u \notag \\
= \, & -\int_{(-\infty, 0]} f(-\mu, t)E(d\mu) \int_{(-\infty, 0]}e^{-\int_s^tf(-\lambda, w) dw} E(d\lambda)u \notag \\
= \, & -f(-A,t) \mathcal{T}_{s,t} u
\end{align}
where we used again \cite[Theorem 11.4]{librobern}.
\end{proof}

\subsection{Time-changed Brownian motion via non-homogeneous subordinators}
In this section we provide some basic facts concerning Brownian motion time-changed with a non-homogeneous subordinator. This is the immediate generalization of the classical subordinate Brownian motion: the reader can consult \cite{bogdan, ksv10, ksv12, sv03} for recent developments on this point.
Therefore we assume now that
\begin{align}
T_t u \, = \, \mathds{E}^xu(B(t)), \qquad t \geq 0,
\end{align}
for $u\in L^2(\mathbb{R}^n)$ where $B$ is an $n$-dimensional Brownian motion starting from $x \in \mathbb{R}^n$. We have therefore the formal representation
\begin{align}
T_tu \, = \, e^{\frac{1}{2}t\Delta }u
\end{align}
where $\Delta$ is the $n$-dimensional Laplace operator such that
\begin{align}
\textrm{Dom}(\Delta) \, = \, \ll u \in L^2(\mathbb{R}^n) : \left\| \Delta u \right\|_{L^2\l \mathbb{R}^n \r} < \infty \rr.
\end{align}
Therefore we get
\begin{align}
\mathcal{T}_{s,t} u \, = \, \mathds{E}^x u \l B \l \sigma^\Pi(t)-\sigma^\Pi(s) \r \r, \qquad 0 \leq s \leq t, u \in L^2 \l \mathbb{R}^n \r.
\end{align}

Consider, for example, the case of a multistable subordinator, where $\Pi (\lambda, t) \, = \, \int_0^t \lambda^{\alpha(s)}ds$ for a suitable choice of $\alpha(s)$ with values in $(0,1)$. Then $\lambda \mapsto \lambda^{\alpha(s)}$ is a Bern\v{s}tein function for each $s \geq 0$, and Theorem \ref{teo41} leads to
\begin{align}
-\l -\Delta  \r^{\alpha(t)}u \, = \, \frac{\alpha (t)}{\Gamma(1-\alpha (t))} \int_0^\infty \l T_su -u \r s^{-\alpha(t)-1} ds
\label{laplfrac}
\end{align}
for a function $u \in \textrm{Dom}(\Delta)$. Note that in this case we have a Brownian motion composed with the multistable subordinator whose increments have characteristic function
\begin{align}
\mathds{E} e^{i\xi B \l \sigma^\Pi(t) -\sigma^\Pi(s)\r} \, = \, e^{-\int_s^t  ( \left\|\xi \right\|^2/2)^{\alpha(w)}dw}.
\end{align}
By following, for example, \cite[Section 3.1]{autostop} the generator \eqref{laplfrac} can be also defined as
\begin{align}
-\l -\Delta  \r^{\alpha(t)}u \, = \, -\frac{1}{(2\pi)^n} \int_{\mathbb{R}^n} e^{-i \xi \cdot x} \, \left\| \xi\right\|^{2\alpha(t)} \, \widehat{u}(\xi) d\xi
\end{align}
with
\begin{align}
\textrm{Dom}\l \l-\Delta \r^{\alpha(t)} \r \, = \, \ll u \in L^2 \l \mathbb{R}^n \r : \int_{\mathbb{R}^n} \left\| \xi \right\|^{2\alpha(t)} \widehat{u}(\xi) d\xi < \infty, \textrm{ for each } t \geq 0 \rr.
\end{align}

In general, we observe that for any non-homogeneous subordinator we can write
\begin{align}
\mathds{E} e^{i\xi B \l \sigma^{\Pi}(t) - \sigma^\Pi (s) \r} \, = \, e^{-\int_s^t f\l \left\| \frac{\xi}{2} \right\|^2,w \r dw}
\end{align}
and we can adapt \cite[Example 4.1.30]{jacob1} to write
\begin{align}
-f\l - \frac{1}{2} \Delta, t\r u \, = \,- \frac{1}{(2\pi)^n} \int_{\mathbb{R}^n} e^{-i \xi \cdot x} \, f\l \left\| \frac{\xi}{2} \right\|^2, t \r \, \widehat{u}(\xi) d\xi
\end{align}
with
\begin{align}
\textrm{Dom} \l f\l -\frac{1}{2} \Delta, t\r  \r \, = \, \ll u \in L^2 \l \mathbb{R}^n \r : \int_{\mathbb{R}^n} f\l \left\| \frac{\xi}{2} \right\|^2,t \r \widehat{u}(\xi) d\xi < \infty, \forall t \geq 0  \rr.
\end{align}

Therefore, we have by Theorem \ref{teo41} the structure of the solution to a sort of diffusion equation
\begin{align}
\frac{d}{dt} q(t) \, = \, -f\l -\frac{1}{2}\Delta,t \r q(t).
\end{align}
\begin{os} \normalfont
\rev{These last remarks have some relationships with Hoh's symbolic calculus (discussed in \cite{hoh1}, \cite[Chapter 6 and 7]{hoh2}) and in particular with \cite{evans}. In full generality one can consider a Fourier symbol of the form
\begin{align}
\varphi (x, \xi) \, = \, f \l  \iota (x,\xi),x \r
\end{align}
where $f: [0, \infty) \times \mathbb{R}^n  \mapsto \mathbb{R}$ is a Bernstein function for each fixed $x \in \mathbb{R}^n$ and $\iota (x,\xi)$ is a symbol in the Hoh's class. Then under some technical assumptions \cite[Theorem 2.4]{evans} it is true that
the (variable order) pseudodifferential operator
\begin{align}
\mathcal{Q}^{(x)}u(x) \, : = \, (2\pi)^{-n/2} \int_{\mathbb{R}^n} e^{ix \cdot \xi}f \l \iota (x,\xi),x \r d\xi
\end{align}
generates a Feller semigroup (on $C_0 \l \mathbb{R}^n \r$). The Fourier symbol of our processes can be of the form $f \l \iota (x,\xi), t \r$ but the dependence on $t$ clearly originates non-homogeneous Markov processes and therefore two-parameter semigroups.
 }
\end{os}

We investigate here the mean square displacement i.e. the quantity
\begin{align}
\mathpzc{M}(t) \, = \, \int_{\mathbb{R}^n} \left\| x \right\|^2 \Pr \ll B \l \sigma^\Pi (t) \r \in dx \rr.
\end{align}
Roughly speaking, a stochastic process is said to have a diffusive asymptotic behaviour when $\mathpzc{M}(t) \sim Ct$ i.e. the mean square displacement grows linearly with time. When $\mathpzc{M}(t) \sim t^\alpha$, $\alpha \in (0,1)$, the process is said to be subdiffusive, while if $\alpha >1$ it is super-diffusive (the reader can consult \cite{metkla1, metkla2} for an overview on anomalous diffusive behaviours). Here it is interesting to note that the mean value of the L\'evy measure, namely $\int_0^{\infty}w \nu (dw,t)$ determines under which conditions the asymptotic behavior is respectively diffusive, sub-diffusive or super-diffusive.
\begin{prop} 
\label{msdprop}
We have the following behaviours.
\begin{enumerate}
\item \label{it1msd} If and only if
\begin{align}
\int_1^\infty w \phi(dw, t) < \infty \textrm{ for } 0\leq t<t_0\leq  \infty
\label{1conv}
\end{align}
it is true that $\mathpzc{M}(t)  < \infty$ for all $t < t_0$
\item Under \eqref{1conv} for $t_0 = \infty$, we have that
\begin{align}
0 <\lim_{t \to \infty} \frac{\mathpzc{M}(t)}{t} = C < \infty \textrm{ if and only if } \lim_{t \to \infty} \int_0^\infty w \nu(dw,t) \, = C
\end{align}
\begin{align}
\lim_{t \to \infty} \frac{\mathpzc{M}(t)}{t} = \infty \textrm{ if and only if } \lim_{t \to \infty} \int_0^\infty w \nu(dw,t) \, = \infty
\end{align}
\item Under \eqref{1conv} for $t_0 = \infty$, if
\begin{align}
\lim_{t \to \infty} \int_0^\infty w \nu(dw, t) = 0
\end{align}
then
\begin{align}
\lim_{t \to \infty} \frac{\mathpzc{M}(t)}{t} =0.
\end{align}
\end{enumerate}
\end{prop}
\begin{proof}
Observe that under \eqref{1conv}
\begin{align}
\mathpzc{M}(t) \, = \, & \int_{\mathbb{R}^n} \left\| x \right\|^2 \int_0^\infty \Pr \ll B(s) \in dx  \rr \, \Pr \ll \sigma^\Pi(t) \in ds \rr \notag \\
= \, & n \, \int_0^\infty s \Pr \ll \sigma^\Pi(t) \in ds \rr  \notag \\
= \, &-n \frac{d}{d\lambda} e^{-\Pi(\lambda, t)} \bigg|_{\lambda = 0} \notag \\
= \, & n\int_0^\infty w \phi(dw, t) < \infty \textrm{ if } t < t_0.
\label{momento}
\end{align}
Observe that the last integral in \eqref{momento} converges only under \eqref{1conv}.
Now note that
\begin{align}
\lim_{t \to \infty} \frac{\mathpzc{M}(t)}{t} \, = \,& \frac{n\int_0^\infty w \phi(dw, t)}{t} \notag \\
= \, &\lim_{t \to \infty} \frac{n\int_0^t \int_0^\infty w\nu(dw, s)ds}{t}
\end{align}
and therefore the proof of Item (2) and (3) is easy to be done.
\end{proof}

A time-change by means of a multistable subordinator leads in this case to a process with $\mathpzc{M} (t) = \infty$ for any $t$ as a consequence of Item \ref{it1msd} of Proposition \ref{msdprop}. Consider now the measure
\begin{align}
\nu(ds, t)  \, = \,  s^{-1}e^{-\alpha (t)s} \, ds
\label{446}
\end{align}
for a function $\alpha (t)>0$ such that A1) and A2) are fulfilled. The associated Bern\v{s}tein functions become, for each $t \geq 0$,
\begin{align}
f(\lambda, t) \, = \, \log \l 1+ \frac{\lambda}{\alpha (t)} \r
\end{align}
and in view \eqref{446} we can compute
\begin{align}
\mathpzc{M}(t) \, = \, n \int_0^t \frac{d\tau}{\alpha (\tau)}  .
\label{msdmultigamma}
\end{align}
Observe that Proposition \ref{msdprop} leads to the study of the limit
\begin{align}
\lim_{t \to \infty} \int_0^\infty e^{-\alpha(t)w}dw \, = \, \lim_{t \to \infty} \frac{1}{\alpha (t)}
\end{align}
therefore the asymptotic behaviour of $\mathpzc{M}(t)$ in this case depends on the asymptotic behaviour of $\alpha (t)$. 

If instead, for functions $\alpha(t)$ strictly between zero and one and $\theta(t)>0$ as in A1) and A2),
\begin{align}
\nu(ds, t) \, = \, \frac{\alpha(t) s^{-\alpha (t)-1} \, e^{-\theta (t) s}}{\Gamma(1-\alpha(t))} ds
\end{align}
then the Bern\v{s}tein functions are a generalization of the Laplace exponent of the relativistic stable subordinator
\begin{align}
f(\lambda, t) \, = \, \l \lambda +\theta(t) \r^{\alpha(t)}-\theta(t)^{\alpha(t)}
\end{align}
and the asymptotic behaviour of the $\mathpzc{M}(t)$ is determined in this case by the limit
\begin{align}
\lim_{t \to \infty} \int_0^\infty \frac{\alpha(t) s^{-\alpha (t)} \, e^{-\theta (t) s}}{\Gamma(1-\alpha(t))} ds \, = \, \lim_{t \to \infty} \alpha(t) \, \theta(t)^{\alpha(t)-1}.
\end{align}
The explicit form of $\mathpzc{M}(t)$ is here
\begin{align}
\mathpzc{M}(t) \, = \, n \int_0^t \frac{\alpha(\tau)}{\theta(\tau)^{1-\alpha(\tau)}} d\tau.
\end{align}
\begin{os} \normalfont
\rev{
For the reader who is familiar with fractional (anomalous) diffusion processes we give here some intuitive flashes of insight. The subordinate Brownian motion, as well as its generalization proposed above is, in general, a non pathwise continuous process. Hence it is not a diffusion. However the subordinate Brownian motion is sometimes included in the class of the so-called fractional diffusions (or anomalous diffusions) since the discontinuity of the sample paths is introduced via a time-change and therefore it develops in an operational time. If, now, we denote such an operational time as $t^\star$ then the generalized subordinate Brownian motion is the process $B \l \sigma^\Pi (t^\star) \r$ where the external time is $t = \sigma^\Pi (t^\star)$. Therefore it must be clear that
the parameters must depend on the operational time $t^\star$ and not on the externally measured time $t$. In the fractional case this dependece must be meant as $t^\star \mapsto \alpha (t^\star)$ as well as in the tempered case we have $t^\star \mapsto \theta (t^\star)$.
}
\end{os}

\section*{Acknowledgements}
Thanks are due to the Referee whose remarks and suggestions have considerably improved a previous draft of the paper.

\vspace{1cm}

\end{document}